\documentclass[11pt]{article}
\usepackage{amsmath}
  \usepackage{paralist}
  \usepackage{graphics}
  \usepackage{epsfig}
  \usepackage{float}
  \usepackage{graphicx}
  \usepackage{abstract}
 \usepackage{epstopdf}
 \usepackage{amssymb}
 \usepackage{cite}
 \usepackage{mathrsfs}
 \usepackage{amsthm}
 \usepackage[colorlinks=true]{hyperref}
\hypersetup{urlcolor=blue, citecolor=red}

\numberwithin{equation}{section}

\newtheorem{theorem}{Theorem}[section]

\newtheorem{physical conclusion}[theorem]{Physical Conclusion}
\newtheorem{remark}[theorem]{Remark}
\newtheorem{assumption}[theorem]{Assumption}

\newtheorem{definition}[theorem]{Definition}

\title{A New Fluid Dynamical Model Coupling Heat with Application to Interior Separations}
\author{\scshape Jiayan Yang \thanks{Email:jiayan\_{}1985@163.com}
\\ \footnotesize Department of Mathematics
\\ \footnotesize Sichuan University
\\ \footnotesize Chengdu,
Sichuan 610064, China
\medskip
\\ \scshape Ruikuan Liu \thanks{Corresponding author:liuruikuan2008@163.com. Supported by NSFC(11401479).}
\\ \footnotesize Department of Mathematics
\\ \footnotesize Sichuan University
\\ \footnotesize Chengdu,
Sichuan 610064, China}

\begin{document}
\date{May 30, 2016}
\maketitle
\begin{abstract}
 Based on the Boussinesq equations and the equation of state, a new fluid dynamical model coupling heat is established. Furthermore, the
 conditions for interior separation, which are determined by the initial conditions, are obtained. This result is derived from the new model by using the interior
 separation theorem which was established by T. Ma and S. Wang in \cite{TM-SW1,TM-SW3}.  The most important application of this
 result is to predict when and where tornado and hurricane  will occur.
\begin{center}
\textbf{\normalsize keywords}
\end{center}
The Boussinesq equations, The equation of state, Interior separation,  Tornado, Hurricane.
\end{abstract}

\section{Introduction}
Separations of fluid flows are fundamental issues in fluid dynamics. The physical and
numerical descriptions of boundary layer separation
 go back to the pioneer work of Prandtl \cite{TM-SW3} in 1904.
Boundary layer separation is the phenomenon that a vortex is generated from the boundary. It is a very common phenomenon
in geophysical dynamics, such as gyres of gulf stream, separation of atmospheric circulation near mountain. There
are many researches \cite{MG-TM-SW,TM-SW2,TM-SW3,LH-WQ-MT,Oleinik,QW-HL-TM,OBL-VAL,PB-MAP,A.Chorin-J.Marsden} on boundary layer separation over the past one
hundred years.
It is noticed that the geometric theory of 2D incompressible flows was initiated by the
authors in \cite{MG-TM-SW,TM-SW2,TM-SW3} to study the structural stability and
transition of 2D incompressible fluid flows. The results give new rigorous
characterization of boundary layer separation(see Ref.\cite{TM-SW2} and references therein). Recently, a predicable condition for boundary layer separation of
2D incompressible fluid flows, which is determined by initial values and external forces, was obtained in \cite{LH-WQ-MT,QW-HL-TM}.

For the interior separation,  the fluid flows can separate from the interior, generating
a vortex.
The kinematic theory for interior separation of 2D incompressible flows was initiated by
T. Ma and S. Wang in \cite{TM-SW1,TM-SW3}. The interior separation problems for
the fluid dynamical equations were discussed in \cite{TM-SW5,TM-SW6}. It is well known that interior separation phenomena correspond to tornado and hurricane  in
geophysical fluid dynamics and  climate dynamics, which are caused by the horizontal thermal motion.

We know that the classical Boussinesq equations \cite{Chandr} are invalid to investigate interior separation because of no horizontal thermal expanding forces,
which are given as follows
\begin{eqnarray}
\label{N-S-H1} \left\{
\begin{aligned}
 &\frac{\partial \mathbf{u}}{\partial t}+(\mathbf{u}\cdot \nabla)\mathbf{u}=\mu \Delta \mathbf{u} - \frac{1}{\rho}\nabla p +(1-\alpha T)g\vec{k} + F,
 \\
  &\frac{\partial T}{\partial t}+(\mathbf{u}\cdot \nabla )T =\kappa \Delta T  + Q,\\
  &\text{div}\mathbf{u} =0,
  \end{aligned}
  \right.
\end{eqnarray}
where $\rho$ is the density of the fluid, $g$ is acceleration due to gravity, $\alpha $ is the coefficient of thermal expansion of the fluid, $\vec{k}$=(0,0,1)
is the unit vector in the $x_3-$direct, $\mu$ is the dynamic viscosity coefficient, $\kappa$ is the thermal diffusion coefficient, Q is the thermal source and $F$ is the external force. The
unknown functions are the velocity field $\mathbf{u}=(u_{1}, u_{2}, u_{3})$, the pressure
function $p$ and the temperature function $T$.

In order to study the interior separation problems by using the fluid dynamical equations coupling horizontal heat, we have
 to modify the Boussinesq equations ($\ref{N-S-H1}$). For this purpose,  we set up a new fluid dynamical model for the fluid flows, which
is given as follows
\begin{eqnarray}
\label{N-S-H2} \left\{
\begin{aligned}
  &\frac{\partial \mathbf{u}}{\partial t}+(\mathbf{u}\cdot \nabla)\mathbf{u}=\mu \Delta \mathbf{u}+\nu \nabla \text{div} \mathbf{u} - \beta T \nabla
  \varphi-\delta\nabla \varphi -\beta\nabla T + \mathbf{F}, \\
  &\frac{\partial T}{\partial t}+ (\mathbf{u}\cdot \nabla)T =\kappa \Delta T + Q,\\
  &\frac{\partial{\varphi}}{\partial t}+ (\mathbf{u}\cdot\nabla) \varphi+\text{div} \mathbf{u} =0,
\end{aligned}
\right.
\end{eqnarray}
where $\beta$ and $\delta$ are constants, $\mu$ and  $\nu$ are the dynamic viscosity coefficients, $\varphi=\ln\rho$ is the density function and $\mathbf{F}$ is the external force.

Applying the kinematic theory of interior separation, which is established  by
T. Ma and S. Wang,  into the model $(\ref{N-S-H2})$, we develop a theory to predict when, where and how the interior separation occurs, and apply the theory to study the tornado and hurricane problems.

The paper is organized as follows. In Section 2, we recall some preliminaries, including  definition of topologically equivalent,  and definitions of boundary layer separation and interior separation, and the
interior separation theorem, and the structural stability theorem.
Section 3 establishes a new fluid dynamical model coupling heat for fluid flows.
In Section 4, we derive interior separation theorem for the new model and
give the predicable conditions for tornado and hurricane. Finally, we give some physical
interpretations to predict when and where the interior separation will occur.

\section{Preliminaries}
Let $M\subset \mathbb{R}^2$ be a $C^r(r\geq 1)$ closed domain with boundary, $C^r(M,\mathbb{R}^2)$  be the space of all $C^{r}$
vector fields on $M$.
Let
\begin{equation}\label{D-r-u}
D^{r}(M,\mathbb{R}^2)=\{\mathbf{w} \in C^{r}(M,\mathbb{R}^2)|  \text{div} \mathbf{w}=0\}.
\end{equation}
\begin{definition}\cite{TM-SW6,TM-SW3}.
Two vector fields $\mathbf{u}$ and $\mathbf{v}\in D^{r}(M,\mathbb{R}^2)$ are called topologically equivalent if there exists
a homeomorphism of $\Upsilon:M\rightarrow M$, which takes the orbits of $\mathbf{u}$ to orbits of $\mathbf{v}$ and preserves their orientation.
\end{definition}

\begin{definition}\cite{TM-SW1}.
A vector field  $\mathbf{v}\in D^{r}(M,\mathbb{R}^2)$ is called structurally stable in $D^{r}(M,\mathbb{R}^2)$ if there exists a
neighborhood $\mathcal{O}\subset D^{r}(M,\mathbb{R}^2)$ of $\mathbf{v}$  such that for any $\mathbf{u}\in\mathcal{O}$, $\mathbf{u}$ and $\mathbf{v}$ are topologically equivalent.

\end{definition}

\begin{definition}\cite{TM-SW6,TM-SW3,TM-SW2}.
We call that boundary layer separation of a 2D vector field $\mathbf{u}$ occurs at $t_0$, if $\mathbf{u}(x, t)$ is topologically equivalent to the structure of
Figure 1(a) for any $t<t_0$, and to the structure of Figure 1(b)
for $t>t_0$. That is, if $t<t_0$, $\mathbf{u}(x, t)$ is topologically equivalent to a parallel flow, and if $t>t_0$, $\mathbf{u}(x, t)$ separates a vortex.
Furthermore,
we call that boundary layer separation occurs at $\overline{x}$, if $\overline{x}$ is an isolated boundary singular point at time $t=t_0$.
\end{definition}

\begin{figure}[H]
  \centering
  \includegraphics[width=1.0\textwidth]{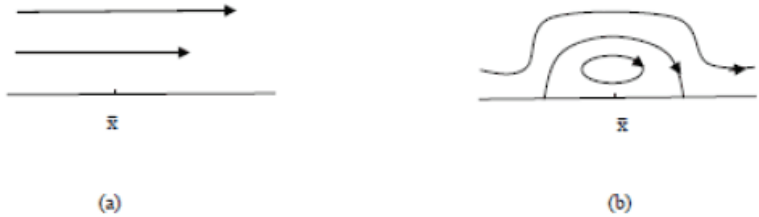}\\
  \caption{Definition of boundary layer separation}\label{B-1}
\end{figure}

\begin{definition}\cite{TM-SW1,TM-SW6,TM-SW3}.
We call that interior separation of a
2D vector field $\mathbf{u}$ occurs at $t_0$, if $\mathbf{u}(x, t)$ is topologically equivalent to the structure of Figure 2(a) for any $t<t_0$, and to the
structure of Figure 2(b)
for $t=t_0$, and to the structure of Figure 2(c) or Figure 2(c$'$)
for $t>t_0$. That is, if $t<t_0$,  the flow given by Figure 2(a) exhibits no singular point
in the neighborhood of $x_0$. At $t=t_0$, $\mathbf{u}(x, t_0)$ is given by Figure 2(b), which has
an isolated singular point $x_0\in \mathring{M}$ with index zero. When $t>t_0$, the flow pattern is given by Figure 2(c) or Figure 2(c$'$) in the back flow region.
\end{definition}

\begin{figure}[H]
  \centering
  \includegraphics[width=1.0\textwidth]{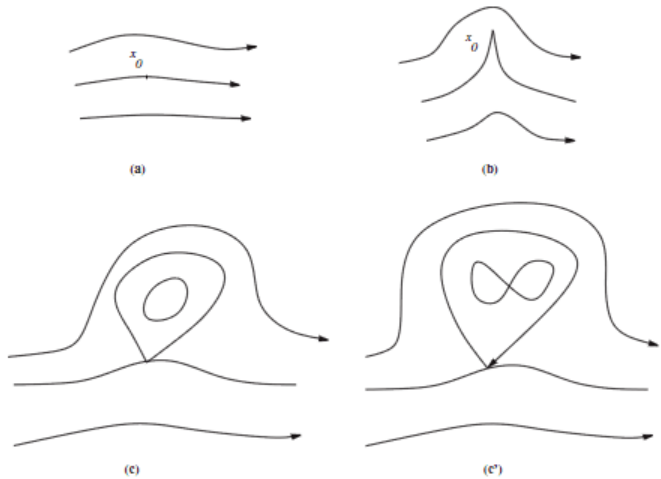}\\
  \caption{Definition of interior separation}\label{B-1}
\end{figure}

Let $\mathbf{u} \in C^{1}([0,\tau],D^{r}(M,\mathbb{R}^2))$$(\tau> 0)$ be a one-parameter family of divergence-free vector fields and its Taylor expansion at
$t=t_{0}(0<t_{0}<\tau)$ is written as
\begin{eqnarray}
\label{N-S-H3} \left\{
\begin{aligned}
 &\mathbf{u}(x,t) = \mathbf{u}^{0}(x) + \mathbf{u}^{1}(x)(t-t_{0}) + o(|t-t_{0}|), \\
  &\mathbf{u}^{0}(x)=\mathbf{u}(x,t_{0}), ~~~x\in M,\\
  &\mathbf{u}^{1}(x)=\frac{\partial \mathbf{u}}{\partial t}\mid_{t=t_{0}}.
  \end{aligned}
  \right.
\end{eqnarray}
For convenience, we denote $D\mathbf{u}^{0}(x_0)$  is the Jacobian matrix of $\mathbf{u}^0(x)$ at $x_0$, which can be expressed as
\begin{eqnarray}
\label{Jacobi}D\mathbf{u}^{0}(x_0)=\left(\begin{array}{cc}
\frac{\partial{u}^{0}_1(x_0)}{\partial x_1} & \frac{\partial{u}^{0}_1(x_0)}{\partial x_2} \\
\frac{\partial{u}^{0}_2(x_0)}{\partial x_1} & \frac{\partial{u}^{0}_2(x_0)}{\partial x_2} \\
\end{array}
\right),
\end{eqnarray}
where $\mathbf{u}^{0}(x_0)$ is defined in $(\ref{N-S-H3})$.

The following assumption of the vector field $\mathbf{u}$ which is crucial in the interior separation theorem.
\begin{assumption}
Let $x_{0}\in \mathring{M}$ be an isolated degenerated singular point of $\mathbf{u}^{0}(x)$. Suppose that
\begin{align}\label{Assumption1-1}
   \text{ind}(\mathbf{u}^{0},x_{0}) &=0, \\
  \label{Assumption1-2} D\mathbf{u}^{0}(x_{0}) &\neq 0,\\
  \label{Assumption1-3} \mathbf{u}^{1}(x_{0})\cdot e_{2} &\neq 0,
\end{align}
where $D\mathbf{u}^{0}(x_{0})$ is the Jacobian matrix of $\mathbf{u}^0(x)$ at $x_0$, $\text{ind}(\mathbf{u}^{0},x_{0})$ is the index of  $\mathbf{u}^{0}$ at the isolated sigular point $x_0$ and $e_{2}$ is as defined in the following $(\ref{Jacobi2})$.
\end{assumption}

Let $x_{0}\in \mathring{M}$ be an isolated degenerated singular point of $\mathbf{u}^{0}(x)$ with nonzero Jacobian: $D\mathbf{u}^0(x_0)\neq 0$. Since $D\mathbf{u}^{0}(x_0)$ is a degenerated nozero matrix, $D\mathbf{u}^{0}(x_0)$ has an eigenvector $e_1$ satisfying
\begin{equation}\label{Jacobi1}
D\mathbf{u}^{0}(x_0)e_1=0,~~~~~|e_1|=1.
\end{equation}
Let $e_2$ be a unit vector, orthogonal to $e_1$, and satisfies that
\begin{equation}\label{Jacobi2}
D\mathbf{u}^{0}(x_0)e_2=\alpha e_1,
\end{equation}
for some constant $\alpha\neq 0$.

We now recall interior  separation theorem and interior  structure bifurcation theorem
established by T. Ma and S. Wang \cite{TM-SW1,TM-SW3} in the following.
\begin{theorem} Let $\mathbf{u} \in C^{1}([0,\tau], D^{r}(M,\mathbb{R}^2))(\tau > 0, r\geq 1)$ satisfy the Assumption 2.5. Then the vector field $\mathbf{u}$ has the interior separation  at
$x_0\in\mathring{M}$ as shown schematically in Figure 2(c) or Figure 2(c$'$).
\end{theorem}

 The following theorem provides necessary and sufficient conditions
for structural stability of a divergence-free vector field in $D^{r}(M,\mathbb{R}^2)$.
\begin{theorem} \cite{TM-SW0,TM-SW1,TM-SW3}.
A divergence-free vector field $\mathbf{u} \in D^{r}(M,\mathbb{R}^2)$
is structurally
stable in $D^{r}(M,\mathbb{R}^2)$ if and only if
\begin{enumerate}
\item u is regular;
\item all interior saddle points of u are self-connected; and
\item each saddle point of u on $\partial M$ is connected only to saddle points on the same connected component of $\partial M$.
\end{enumerate}
Moreover, the set of all structurally stable vector fields is open and dense in $D^{r}(M,\mathbb{R}^2)$.
\end{theorem}
\begin{remark}
 When $M\subset \mathbb{R}^2$ is without boundary, the results similar to Theorem 2.6 hold true as well.
\end{remark}

\section{A new fluid dynamical model coupling heat}
Based on the Newton's second law, the motion of fluid flows can be governed by the Navier-Stokes equations coupling heat expressed as
\begin{equation}\label{N-S3}
  \frac{\partial \mathbf{u}}{\partial t}+(\mathbf{u}\cdot \nabla)\mathbf{u}=\mu \Delta \mathbf{u} +\nu \nabla \text{div} \mathbf{u}- \frac{1}{\rho}\nabla p
  +(1-\alpha T)g\vec{k} + \mathbf{F},
\end{equation}
where $\mathbf{u}=(u_1,u_2,u_3)$ is the velocity field, $p$ is the pressure,  $g$ is the acceleration due to gravity, $\alpha $ is the coefficient of thermal expansion
of the fluid, $\vec{k}$=(0,0,1) is the unit vector in the $x_3-$direct, $\rho$ is the mass density, $\mu$,$\nu$ are the dynamic viscosity  constants and $\mathbf{F}$ is
the external force.

The conservation of mass  takes the following form
\begin{equation}\label{Conservation1}
  \frac{\partial \rho}{\partial t}+ \text{div} \rho\mathbf{u}=0,
\end{equation}
which, with the constant density, is reduced to $\text{div}\ \mathbf{u}=0.$

It is well known that  different fluid flows correspond to different expressions for equation of state. When it comes to gas, the fluid is approximated as the
ideal gas,  which is enough in  gas dynamical applications  in \cite{ANDER} and the equation of state is given by
 \begin{align}\label{IDEALGAS}
   p\mathscr{V}=\mathscr{NR}T,
    \end{align}
where $\mathscr{N}$ is the number of moles of gas, $\mathscr{V}$ is
the volume of the system, $\mathscr{R} = 8314J/(kg\cdot mol\cdot K)$ is the gas constant, $p$ is the pressure and $T$ is  the temperature.

It is easy to see that  $(\ref{IDEALGAS})$ can be rewritten as follows
\begin{align}
  p&= \frac{\mathscr{NR}T}{\mathscr{V}} \nonumber\\
   &=\frac{\mathscr{R}}{M}\cdot\frac{\mathscr{N}M}{\mathscr{V}}T\nonumber\\
  \label{state} &=R\rho T,
\end{align}
where $M$ is the molecular weight, $R=\frac{\mathscr{R}}{M}$ is specific gas constant depending on the type of the gas and $\rho$ is the density.

When the fluid flow is liquid, the relationship between the pressure $p$ and  the temperature $T$ can be approximated by
\begin{equation}\label{PT-INCOMP}
  p =\rho(\sigma T + \gamma),
\end{equation}
where $\sigma, \gamma$ are constants.

Combining $(\ref{state})$ and  $(\ref{PT-INCOMP})$, we conclude that
\begin{align}\label{PT-GENERAL}
  p=\rho(\beta T + \delta),
\end{align}
where $\beta$ and $\delta$ are depending on the types of fluid flows, i.e. $$
 \beta=\left \{ \begin{aligned}
    & R, \quad \quad  \text{for}\quad \text{gas},\\
    &\sigma, \quad \quad  \text{for} \quad \text{liquid},
    \end{aligned}
     \right.$$
and
     $$
 \delta=\left \{ \begin{aligned}
    & 0, \quad \quad  \text{for}\quad \text{gas},\\
    &\gamma, \quad \quad  \text{for} \quad \text{liquid}.
    \end{aligned}
     \right.$$

Let the density $\rho$ of the fluid flows be
\begin{equation}\label{density}
  \rho (x,t)= e^{\varphi(x,t)},
\end{equation}
where $x \in \mathbb{R}^n(n=2,3)$, $t \in [0,\infty)$. Note that $\varphi=\ln\rho$. Moreover, $-\frac{1}{\rho}\nabla p$ in $(\ref{N-S3})$  can be rewritten as
\begin{align}\label{N-B2}
  -\frac{1}{\rho}\nabla p=-\frac{1}{\rho}\nabla(\rho(\beta T + \delta))=-\beta T\nabla \varphi-\beta\nabla T-\delta\nabla \varphi.
\end{align}
For simplicity, we denote $\mathbf{F}=F+(1-\alpha T)g\vec{k}$, $(\ref{N-S3})$
and $(\ref{N-B2})$ imply that
\begin{align}\label{N-S-H4}
  \frac{\partial \mathbf{u}}{\partial t}+(\mathbf{u}\cdot \nabla)\mathbf{u}&=\mu \Delta \mathbf{u}+\nu \nabla \text{div} \mathbf{u}-(\beta T+\delta)\nabla
  \varphi -\beta\nabla T +\mathbf{F}.
\end{align}

From the Heat Conduction law, we obtain that
\begin{align}\label{N-S-H5}
 \frac{\partial T}{\partial t}+(\mathbf{u}\cdot \nabla )T  =\kappa \Delta T + Q,
\end{align}
where $\kappa$ is thermal diffusion coefficient and Q is the thermal source.

Together $(\ref{Conservation1})$ with $(\ref{density})$, we also get
\begin{align}\label{conservation2}
  \frac{\partial \varphi}{\partial t}
  +(\mathbf{u}\cdot\nabla)\varphi+\text{div}\mathbf{u}=0.
\end{align}

From $(\ref{N-S-H4})-(\ref{conservation2})$, we immediately drive the model $(\ref{N-S-H2})$.

\begin{remark}
    The model $(\ref{N-S-H2})$ can govern the motion of  2D or 3D fluid flows for both compressible and incompressible.
  It is obvious that the model $(\ref{N-S-H2})$ contains the horizontal thermal expanding force
  $-\beta T \nabla
 \varphi -\beta\nabla T $.
\end{remark}
\begin{remark}
     Instead of the previous  pressure term  $p$ in  $(\ref{N-S-H1})$, our model $(\ref{N-S-H2})$ only relates to temperature $T$, which is crucial to study the natural phenomena(tornado, hurricane) caused by the horizontal thermal expanding forces.
\end{remark}

\section{Main Results}
We firstly study the interior separation for  2D incompressible fluid flows with the horizontal thermal motion.
\subsection{Interior separation theorem for 2D incompressible fluid flows coupling heat}
 In large scale, the density $\rho$ of  fluid flows can be approximatively seen as a constant. In this case, $\varphi$ is also a constant from
 $(\ref{density})$.  Then, the equations (\ref{N-S-H2}) are  rewritten as
\begin{eqnarray}
\label{N-H1} \left\{
\begin{aligned}
  \frac{\partial \mathbf{u}}{\partial t}+(\mathbf{u}\cdot \nabla)\mathbf{u}  &=\mu \Delta \mathbf{u}-\beta\nabla T + \mathbf{F}, \\
  \frac{\partial T}{\partial t}+ (\mathbf{u}\cdot \nabla )T &=\kappa \Delta T  + Q,\\
  \text{div}\mathbf{u} &=0,
\end{aligned}
\right.
\end{eqnarray}
where $\mathbf{u}=(\mathbf{u}_{1},\mathbf{u}_{2})$ is the  velocity field, $T$ is the temperature, $\mu$ is the  dynamic viscosity coefficient, $\kappa$ is the
thermal diffusion coefficient, $Q$ is the thermal source and $\mathbf{F}$ is the external force
with the initial condition
\begin{align}
\label{FF1}\mathbf{F}\mid_{t=0}&= F^{0}(x).
\end{align}
The corresponding initial conditions  of $(\ref{N-H1})$ are given by
\begin{eqnarray}
\label{UU1} \left\{
\begin{aligned}
  \mathbf{u}\mid_{t=0}&= \Psi(x), \\
  T\mid_{t=0}&= T^{0}(x).
\end{aligned}
\right.
\end{eqnarray}

A natural method for studying interior separation of such a family $\mathbf{u}(\cdot,t)$ and $\mathbf{T}(\cdot,t)$ are to Taylor expand them near initial time $t=0$, and then to
analyze their
structure in a neighborhood of $t=0$, which is given  by
\begin{eqnarray}
\label{U-1} \left\{
\begin{aligned}
 &\mathbf{u}(x,t) = \Psi(x) + t \mathbf{u}^{1}(x)+ o(|t|), \\
 &T(x,t) =T^{0}(x) + tT^{1}(x) + o(|t|),\\
  &\Psi(x)=\mathbf{u}(x,0),~~\mathbf{u}^{1}(x)=\frac{\partial \mathbf{u}}{\partial t}\mid_{t=0},\\
  &T^{0}(x)=T(x,0),~~T^{1}(x)=\frac{\partial T}{\partial t}\mid_{t=0}.
  \end{aligned}
  \right.
\end{eqnarray}
From $(\ref{N-H1})$-$(\ref{U-1})$, we deduce that
\begin{align}\label{U-11}
  \mathbf{u}^{1}(x)=\frac{\partial \mathbf{u}}{\partial t}\mid_{t=0}= \mu \Delta\Psi - (\Psi\cdot\nabla)\Psi-\beta \nabla T^{0} + F^0,
\end{align}
which obviously implies that
\begin{align}\label{U-12}
  \mathbf{u}(x,t) = \Psi(x) + t[\mu \Delta\Psi - (\Psi\cdot\nabla)\Psi-\beta \nabla T^{0}+ F^0] + o(|t|).
\end{align}
Let
\begin{equation}\label{U-13}
     \mathbf{v}(x,t) = \Psi(x) + t[\mu \Delta\Psi - (\Psi\cdot\nabla)\Psi-\beta\nabla T^{0}+ F^0].
   \end{equation}
It is obvious that
\begin{equation}\nonumber
     \mathbf{u}(x,t)=\mathbf{v}(x,t)+o(|t|)\ \ \ \text{for} \ \ 0\leq t\ll 1.
   \end{equation}
%
Now, we make the following assumption.
\begin{assumption}
Let the  functions
$\Psi,T^{0},F^0$,  given by $(\ref{FF1})$ and $(\ref{UU1})$, be the solutions of the following equations
\begin{align}
   \label{AS2} \text{div}~\Psi(x)= 0&, \ \ \ for \ \ \forall\  x\in \mathbb{R}^2, \\
  \label{AS3}
       \frac{\partial \Psi_{1}}{\partial x_{2}}\frac{\partial \Psi_{2}}{\partial x_{1}}+(\frac{\partial \Psi_{1}}{\partial x_{1}})^{2} +\frac{\beta}{2} \Delta
       T^{0} -\frac{1}{2}\text{div}F^0 =0&,\ \ \ for \ \ \forall\  x\in \mathbb{R}^2,
        \end{align}
and they satisfy the conditions that there exists an
$\bar{x}\in U\subset \mathbb{R}^2$ such that
$\bar{x}$ is an isolated singular point of $\mathbf{v}(x, t_0)$ with the Jacobian matrix
$D\mathbf{v}(\bar{x},t_0)\neq 0$,  and there exists a time
$t_0$ with $0\leq t_0\ll 1$ such that
the vector $\mathbf{v}(x,t)\neq 0$ for $x\in U,~0\leq t< t_0$, and
\begin{equation}\label{vector-v}
[\mu \Delta\Psi - (\Psi\cdot\nabla)\Psi-\beta\nabla T^{0}+ F^0]\mid_{x=\bar{x}}\cdot e_2\neq 0,
\end{equation}
where $\mathbf{v}(x,t)$ is as given by  $(\ref{U-13})$ and $e_{2}$ is as defined in $(\ref{Jacobi2})$.
\end{assumption}
The interior separation theorem of the model $(\ref{N-H1})$ is given by the following theorem.
\begin{theorem}
 Let
 $\mathbf{u}(x,t), T(x,t) $ be the solution of equations $(\ref{N-H1}),~(\ref{UU1})$. If the functions $\Psi,T^{0}, F^{0}$ satisfy Assumption 4.1, then
 the vector
 field $\mathbf{u}$ has an interior separation near $(\bar{x},t_{0})$.
\end{theorem}
\begin{proof}

First, we show that the vector field $\mathbf{v}(x,t)$ has an interior separation near $(\bar{x},t_{0})$.
From $(\ref{U-13})-(\ref{AS3})$, we get
\begin{align}\label{divergence}
  \text{div}\mathbf{v}&=\text{div}\{\Psi(x) + t[\mu \Delta\Psi - (\Psi\cdot\nabla)\Psi-\beta\nabla T^{0}+ F^0]\}\nonumber\\
  &=-t\{\text{div}(\Psi\cdot\nabla)\Psi+\beta\Delta T^{0}-\text{div}F^0\}\nonumber\\
  &=-t\{2\frac{\partial \Psi_{1}}{\partial x_{2}}\frac{\partial \Psi_{2}}{\partial x_{1}}+2(\frac{\partial \Psi_{1}}{\partial x_{1}})^{2} +\beta\Delta
       T^{0} -\text{div}F^0\}\nonumber\\
  &=0.
 \end{align}

By Theorem 2.6, we only need to prove that the divergence-free $\mathbf{v}(x,t)$ satisfies
the following.
From Assumption 4,1, we know that  $U$ is the neighborhood of the singular point $\bar{x}$ and
$$\mathbf{v}(x,t)\neq 0 ~~~~~~ \text{for}~~~ ~~x\in U,~~0\leq t< t_0,$$
which implies
\begin{equation}\label{Ind}
    \text{ind}(\mathbf{v}(x,t_0),\bar{x})=0,\ \ \ \  \text{det}D\mathbf{v}(\bar{x},t_0)=0.
  \end{equation}
Moreover,  it is easy to see that that the Jacobian matrix
$D\mathbf{v}(\bar{x},t_0)\neq 0$ and
\begin{equation}\label{vect-1}
\frac{\partial \mathbf{v}}{\partial t}\mid_{x=\bar{x},t=t_0}\cdot e_2= [\mu \Delta\Psi - (\Psi\cdot\nabla)\Psi-\beta\nabla T^{0}+ F^0]\mid_{x=\bar{x}}\cdot e_2\neq 0.
\end{equation}
From the above discussion, we see that the Assumption 2.5 is satisfied. So, we can obtain that
the vector $\mathbf{v}$ has an interior separation near $(\bar{x},t_{0})$.

Second, we will prove that the velocity $\mathbf{u}$ has an interior separation near
$(\bar{x},t_{0})$.
By $(\ref{U-12})$ and $(\ref{U-13})$, we see that
\begin{equation}\label{vec11}
\mathbf{u}(x,t)=\mathbf{v}(x,t)+o(|t|).
\end{equation}

Diving by $t$ on the both sides of $(\ref{vec11})$, we obtain
\begin{equation}\label{vect-2}
\frac{1}{t}\mathbf{u}(x,t)=\frac{1}{t}\mathbf{v}(x,t)+O(|t|),~~\big(O(|t|)\rightarrow 0\ \  \ \text{as}\ \ \ t\rightarrow 0\big).
\end{equation}
It is easy to see that  $\exists$ $\tau$ with $0<t<\tau\ll 1$ such that for $t\in(0,\tau)$
\begin{equation}\label{vect-3}
\frac{1}{t}\mathbf{u}\ \ \text{is\ a \ perturbation\  of} \ \ \frac{1}{t}\mathbf{v}.
\end{equation}
Due to the Theorem 2.7 and Remark 2.8,  we know that $\frac{1}{t}\mathbf{u}$
and  $\frac{1}{t}\mathbf{v}$  are topologically equivalent for $t_0<t<\tau$.
\\
On the other hand,  $\frac{1}{t}\mathbf{v}$ has the same orbit structure as $\mathbf{v}$. It is noticed that $\mathbf{v}$ has an interior separation near $(\bar{x},t_0)$ in the first discussion. Hence, $\frac{1}{t}\mathbf{u}$ has the interior separation near $(\bar{x},t_0)$, which implies
$\mathbf{u}$ has the interior separation near $(\bar{x},t_0)$.

\end{proof}
\begin{remark}
It is easy to see that $(\ref{AS2})$ and $(\ref{AS3})$
    guarantee $\text{div}\mathbf{v}=0$.
    By the topology degree theory,  the vector
    field $\mathbf{v}(x,t)$ has no zero point in $U$ for $0\leq t< t_0$ implies that   $\text{ind}~(\mathbf{v}(x,t_0),\bar{x})=0$.
\end{remark}

\begin{remark}
The conditions given in  Assumption 4.1  are determined by the initial external force $(\ref{FF1})$ and the initial conditions $(\ref{UU1})$. It is easy to see that
$\mathbf{v}$ has an interior separation near $(\bar{x},t_0)$ is enough to judge the  interior separation of $\mathbf{u}$. This theorem provides an approach to predict
the interior separation depending on the initial functions. More specially, this result  provides an approach to predict the tornado or the hurricane just by
initial velocity, temperature and external force under the Assumption 4.1.
\end{remark}


\subsection{The predictable conditions for  the interior separation}
The main objective of this subsection is to give a specific application of the Theorem 4.2, which are depending on the  initial functions. We consider the model $(\ref{N-H1})$  with initial conditions given by

\begin{eqnarray}
\label{N-H11} \left\{
\begin{aligned}
  &\frac{\partial \mathbf{u}}{\partial t}+(\mathbf{u}\cdot \nabla)\mathbf{u}=\mu \Delta \mathbf{u}-\beta\nabla T + \mathbf{F}, \\
   &\frac{\partial T}{\partial t}+ (\mathbf{u}\cdot \nabla )T=\kappa \Delta T  + Q,\\
  &~~\text{div}\mathbf{u} =0,\\
  &~~\mathbf{u}\mid_{t=0}= \Psi(x),~~T\mid_{t=0}= T^{0}(x),
\end{aligned}
\right.
\end{eqnarray}
where $\mathbf{u}=(\mathbf{u}_{1},\mathbf{u}_{2})$ is the velocity field, $T$ is the temperature function, $\mu$ is the  dynamic viscosity coefficient, $\kappa$ is the
thermal diffusion coefficient, $Q$ is the thermal source and $\mathbf{F}$ is the external force.

To make the equations  nondimensional, let
\begin{align}
&(x,t)=(Lx',\frac{L^2}{\mu}t'),\nonumber\\
&(\mathbf{u},T)=(\frac{\mu}{L}\mathbf{u}',\theta T'),\nonumber\\
&\mathbf{F}=\frac{\mu^2}{L^3}\mathbf{F}', Q=\frac{\mu\theta}{L^2}\mathbf{Q}',\nonumber\\
&\Psi=\frac{\mu}{L}\Psi', T^{0}=\theta T^{'0}.\nonumber
\end{align}
Here $\mu$ is the  dynamic viscosity coefficient, $L$ is the diameter of the neighborhood
of singular point $\bar{x}$ and $\theta$ is the unit of temperature.

Omitting the primes, the model $(\ref{N-H11})$ can be rewritten  as

\begin{eqnarray}
\label{N-H12} \left\{
\begin{aligned}
  &\frac{\partial \mathbf{u}}{\partial t}+(\mathbf{u}\cdot \nabla)\mathbf{u}=\Delta \mathbf{u}-\frac{L^2\beta\theta}{\mu^2}\nabla T + \mathbf{F}, \\
   &\frac{\partial T}{\partial t}+ (\mathbf{u}\cdot \nabla )T=\frac{\kappa}{\mu} \Delta T  + Q,\\
  &~~\text{div}\mathbf{u} =0,\\
  &~~\mathbf{u}\mid_{t=0}= \Psi(x),~~T\mid_{t=0}= T^{0}(x).
\end{aligned}
\right.
\end{eqnarray}

The following theorem is the application of the Theorem 4.2 under nondimensional initial functions.
\begin{theorem} Let $\mathbf{u}(x,t)$, $T(x,t)$ be the solution of equations $(\ref{N-H12})$. If the initial conditions $\Psi(x),T^{0}$ satisfy
  \begin{eqnarray}
\label{examp1}\left\{
\begin{aligned}
    &\Psi=(\Psi_{1},\Psi_{2})=(0,1+C_1x_{1}^{2}),\\
         &T^{0}=C_2+C_3x_{2},
   \end{aligned}
  \right.
    \end{eqnarray}
and  the initial external force satisfies
\begin{equation}\label{External}
\mathbf{F}^0 = (C_4x_{2},0),
\end{equation}
where  $C_i(i=1,2,3,4)$ are nonzero constants. Then vector field $\mathbf{u}$  exists an interior separation
near $(0,t_{0})(t_{0}=\frac{1}{\frac{L^2\beta\theta}{\mu^2}C_3-2C_1}$ with $\frac{L^2\beta\theta}{\mu^2} C_3-2 C_1 \gg 1$) and $\mu$ is diffusion coefficient.
\end{theorem}

\begin{proof} From $(\ref{examp1})$ and $(\ref{External})$, it is obviously  that the functions $\Psi(x),T^{0},\mathbf{F}^0$ are the solutions of $(\ref{AS2})$ and $(\ref{AS3})$.
By $(\ref{U-13})$,  $(\ref{examp1})$ and $(\ref{External})$, we know that
\begin{eqnarray}
\label{Singular}
\left\{
\begin{aligned}
&\mathbf{v}_1(x,t_0)=C_4x_2t_0=0 ,\\
&\mathbf{v}_2(x,t_0)=1+t_0[{2 C_1-\frac{L^2\beta\theta}{\mu^2} C_3}]-C_1x^2_1t_0=0.
 \end{aligned}
  \right.
    \end{eqnarray}
When $t=t_0=\frac{1}{\frac{L^2\beta\theta}{\mu^2} C_3-2 C_1}$ in  $(\ref{Singular})$,  we immediately derive that $(0,0)$ is the isolated singular point of $\mathbf{v}(x,t_0)$.

Combining $(\ref{U-13})$ and $\Psi(\bar{x})=\Psi(0)\neq0$, we see that there exists an open set $U\subset\mathbb{R}^2$ and a time
$t_0$ with $0\leq t_0\ll 1$ such that
the vector $\mathbf{v}(x,t)\neq 0$ for $x\in U,~0\leq t< t_0$.

Next, we will prove the Jacobian matrix $D\mathbf{v}(0,t_0)$ of $\mathbf{v}(0,t_0)$ is nonzero degenerated matrix.
By the simple calculation from  $(\ref{Singular})$, we get
\begin{eqnarray}
\label{Jacobi3} D\mathbf{v}(0,t_0)=\left(\begin{array}{cc}
0 & C_4t_0 \\
0 & 0 \\
\end{array}
\right).
\end{eqnarray}
 It is obvious that $\text{det}D\mathbf{v}(0,t_0)=0$. Hence, $(0,0)$ is an isolated degenerated singular point of
$\mathbf{v}(x,t_0)$ with $D\mathbf{v}(0,t_0)\neq0$.

It is clear that
\begin{eqnarray}\nonumber
\frac{\partial\mathbf{v}}{\partial t}\mid_{x=\bar{x}=0,t=t_0}\cdot e_2&=&[\mu \Delta\Psi - (\Psi\cdot\nabla)\Psi-\frac{L^2\beta\theta}{\mu^2}\nabla T^{0}+ F^0]\mid_{x=\bar{x}=0}\cdot e_2\\
\nonumber
&=&(0,2 C_1-\frac{L^2\beta\theta}{\mu^2} C_3)\cdot(0,1)\\
&=&2 C_1-\frac{L^2\beta\theta}{\mu^2}C_3 \neq0.
\end{eqnarray}
Therefore, the Assumption 4.1 is satisfied, the conclusion follows from the Theorem 4.2.
\end{proof}

\subsection{Physical Interpretations}
In the following, we give the physical interpretation of the tornado and hurricane related to Theorem 4.6.

{\bf Physical Interpretation 4.7.}
It is easy to see that the tornado and hurricane occurs near $(0,t_0)$ ($t_{0}=\frac{1}{\frac{L^2\beta\theta}{\mu^2}C_3-2C_1}$ is small enough) under the assumptions  $(\ref{examp1})$ and $(\ref{External})$.
It is also noticed that $t_{0}=\frac{1}{\frac{L^2\beta\theta}{\mu^2}C_3-2C_1}$($C_1$ and $C_3$ are nonzero constant) with $\frac{L^2\beta\theta}{\mu^2} C_3-2 C_1 \gg 1$ is the time at which tornado and hurricane  occur. The singular point $\bar{x}$, which can be obtained by $(\ref{Singular})$, is the place where the tornado and hurricane  occur.

{\bf Physical Interpretation 4.8.}
In reality, the diameter $L$ of the neighborhood of the tornado's and hurricane's center $\bar{x}$ is very large.
Although the dimensionless time $t_{0}=\frac{1}{\frac{L^2\beta\theta}{\mu^2}C_3-2C_1}$($C_1$ and $C_3$ are nonzero constant) with $\frac{L^2\beta\theta}{\mu^2} C_3-2 C_1 \gg 1$  is small, we can easy to see that the real predictable time $\bar{t}=\frac{L^2}{\mu}t_0$ is not small.

{\bf Physical Interpretation 4.9.}
From $\frac{L^2\beta\theta}{\mu^2} C_3-2 C_1 \gg 1$, we deduce that
\begin{equation}\label{me}
\frac{L^2\beta\theta}{\mu^2} C_3\gg 1
\end{equation}
or
\begin{equation}\label{me2}
-2 C_1 \gg 1.
\end{equation}
It is noticed that $(\ref{me})$ implies that large difference in temperature can cause
tornado and hurricane. It is also noticed that $(\ref{me2})$ means that the large initial velocity can cause tornado and hurricane.  More importantly, the meaning of  $(\ref{me})$ and $(\ref{me2})$ accord with the physical fact.

 {\footnotesize

\end{document}